\documentclass[journal]{IEEEtran}
\usepackage{mathrsfs,amsmath,amsthm,amssymb,amsfonts}
\usepackage{stmaryrd}
\usepackage{algorithm,algorithmic}
\usepackage{subfigure}
\usepackage{verbatim}
\usepackage{cite}
\usepackage{blindtext}
\usepackage{esint}
\usepackage{graphicx,xcolor,subfig}
\usepackage{epstopdf}
\usepackage[section] {placeins}
\usepackage{morefloats}
\usepackage{url}

\newtheorem{theorem}{Theorem}
\newtheorem{remark}{Remark}[section]
\newtheorem{assumption}{Assumption}[section]
\newtheorem{lemma}[theorem]{Lemma}

\def\calA{{\mathcal{A}}}

\begin{document}

\title{Iterative Soft/Hard Thresholding with Homotopy Continuation for Sparse Recovery}
\author{ Yuling~Jiao, Bangti Jin
         and~Xiliang~Lu
\thanks{Yuling Jiao is in the School of Statistics and Mathematics and Big Data Institute of ZUEL, Zhongnan University of Economics and Law, Wuhan, 430063, P.R. China (email: yulingjiaomath@whu.edu.cn), Bangti Jin is in the Department of Computer Science, University College London, Gower Street, London WC1E 6BT, UK (email: bangti.jin@gmail.com, b.jin@ucl.ac.uk), and Xiliang Lu (corresponding author) is in the School of Mathematics and Statistics, Wuhan University and
Hubei Key Laboratory of Computational Science, Wuhan University, Wuhan 430072, P.R. China (email: xllv.math@whu.edu.cn).}}

\markboth{IEEE SIGNAL PROCESSING Letters,Vol.~~, No.~~, ~~,2015}%
{Shell \MakeLowercase{\textit{et al.}}: Bare Demo of IEEEtran.cls for Journals}
\maketitle

\begin{abstract}
In this note, we analyze an iterative soft / hard thresholding algorithm with homotopy
continuation for recovering a sparse signal $x^\dag$ from noisy data of a noise level $\epsilon$.
Under suitable regularity and sparsity conditions, we design a path along which
the algorithm can find a solution $x^*$ which admits a sharp reconstruction error
$\|x^* - x^\dag\|_{\ell^\infty} = O(\epsilon)$ with an iteration complexity
$O(\frac{\ln \epsilon}{\ln \gamma} np)$, where $n$ and $p$ are problem dimensionality and $\gamma\in (0,1)$ controls
the length of the path. Numerical   examples are given
to illustrate its performance.
\end{abstract}

\IEEEpeerreviewmaketitle

\section{Introduction}\label{sec:intro}

\IEEEPARstart{S}{parse}
recovery has attracted much attention in machine learning, signal processing,
statistics and inverse problems over the last decade. Often the problem is formulated as
\begin{equation}\label{eqn:gov}
  y = \Psi x^{\dag} + \eta,
\end{equation}
where $x^{\dag} \in \mathbb{R}^{p}$ is the unknown sparse signal,
$y\in\mathbb{R}^n$ is the data with the noise $\eta \in\mathbb{R}^{n} $ of level $\epsilon=\|\eta\|$, and
the matrix $\Psi\in\mathbb{R}^{n\times p}$ with $p\gg n$ has normalized columns $\{\psi_i\}$, i.e.,
$ \|\psi_i\| =1$, $i=1,\ldots,p.$
The desired sparsity structure can be enforced by either the $\ell^0$ or $\ell^1$ penalty, i.e.,
\begin{equation}\label{reglasso}
 \min_{x \in \mathbb{R}^{p}}  \tfrac{1}{2}\|\Psi x
 -y\|^{2} + \lambda\|x\|_{t},\quad t\in \{0,1\},
\end{equation}
where $\lambda>0$ is the regularization parameter. 

Among existing algorithms for minimizing \eqref{reglasso}, iterative soft / hard thresholding (IST/IHT)
algorithm \cite{DaubechiesDefrisedeMol:2004, pfbs,BlumensathDavies:2008,Attouch:2013} 
and their accelerated extension  \cite{fista,nesta} are extremely popular. These algorithms are of the form
\begin{equation}\label{equ:ith1}
x^{k+1} = T_{\tau_k\lambda} (x^k + \tau_k \Psi^t (y - \Psi x^k)),
\end{equation}
where $\tau_k$ is the stepsize, and $T_\lambda$ is a soft- or hard-thresholding operator defined componentwise by
\begin{equation}\label{eqn:thresholding}
  T_\lambda(t)=\left\{\begin{array}{ll}
    \max(|t|-\lambda,0)\mathrm{sgn}(t), & \mbox{IST},\\
    \chi_{\{|t|>\sqrt{2\lambda}\}}(t), & \mbox{IHT},
  \end{array}\right.
\end{equation}
where $\chi(t)$ is the characteristic function.
Their convergence was analyzed in many works, mostly under
the condition $\tau_k<2/\|\Psi\|^2$. This condition ensures a (asymptotically) contractive thresholding
and thus the desired convergence \cite{DaubechiesDefrisedeMol:2004,pfbs,BlumensathDavies:2008,Attouch:2013}.
Meanwhile, it was observed that the continuation along $\lambda$ can
greatly speed up the algorithms \cite{fpc,gpsr,sparsa,nesta,Lorenz:2013}.
Nonetheless, as pointed out by \cite{TroppWright:2010} ``... the design of a robust, practical, and theoretically
effective continuation algorithm remains an interesting open question ...'' There were several works aiming
at filling this gap. In the works \cite{XiaoZhang:2013,Agarwal:2012},
a proximal gradient method with continuation for  $\ell^1$ problem was analyzed with linear search,
under sparse restricted eigenvalue/restricted strong convexity condition. Recently, a Newton type
method with continuation was studied for $\ell^1$ and $\ell^0$ problems \cite{FanJiaoLu:2014,JiaoJinLu:2015}.
In this work, we present a unified approach to analyze IST/IHT
with continuation and a fixed  stepsize $\tau =1$, denoted by ISTC/IHTC. The challenge
in the analysis is the
lack of monotonicity of function values due to the choice $\tau =1$.

The overall procedure is given in Algorithm \ref{alg:ithc}. Here $\lambda_0$
is an initial guess of $\lambda$, supposedly large, $\gamma\in(0,1)$ is the decreasing
factor for $\lambda$, and $K_{max}$ is the maximum number of inner iterations (for a
fixed $\lambda$). The choice of the final $\lambda^*$ is given in \eqref{eqn:lambda} below.
Distinctly, the inner iteration
does not need to be solved exactly (actually one inner iteration suffices the desired accuracy
of the final solution $x^*$, cf. Theorem \ref{thm:consoft} below), and there is no need
to perform stepsize selection. 

\begin{algorithm}[hbt!]
   \caption{Iterative Soft/Hard-Thresholding with Continuation (ISTC/IHTC)}\label{alg:ithc}
   \begin{algorithmic}[1]
     \STATE Input: $\Psi\in \mathbb{R}^{n\times p}$, $y$, $\lambda_0$, $\gamma \in (0,1)$, $\lambda^*$, $K_{max}\in \mathbb{N}$, $x(\lambda_0) = 0$.
     \FOR {$\ell=1,2,...$}
     \STATE Let $\lambda_\ell = \gamma \lambda_{\ell-1}$, $x^0 = x(\lambda_{\ell-1})$.
     \STATE If $\lambda_\ell < \lambda^*$, stop and output $x^* = x^0$.
     \FOR {$k = 0,1,...,K_{max}-1$}
     \STATE $x^{k+1} = T_{\lambda_\ell} (x^k + \Psi^t (y - \Psi x^k)).$
     \ENDFOR
     \STATE Set $x(\lambda_\ell) = x^{K_{max}}$
     \ENDFOR
   \end{algorithmic}
\end{algorithm}

In Theorem \ref{thm:consoft}, we prove that under suitable mutual coherence condition on the matrix
$\Psi$ (cf. Assumption \ref{ass:mc} and Remark \ref{rmk:mc}), ISTC/IHTC always converges.

\section{Convergence analysis}

The starting point of our analysis is the next lemma.
\begin{lemma}\label{lem:thresholding}
For any $x,y\in \mathbb{R}$, there holds
\begin{equation*}
|T_{\lambda}(x + y) - x| \leq \left\{\begin{array}{ll}
|y| + \lambda & \mbox{ IST},\\
|y| + \sqrt{2\lambda} & \mbox{ IHT}.
\end{array}\right.
\end{equation*}
\end{lemma}
\begin{proof}
By the definition of the operator $T_\lambda$, cf. \eqref{eqn:thresholding},
\begin{equation*}
  \begin{aligned}
    |T_\lambda(x + y) - x|& \leq |T_\lambda (x+y) - (x+y)| + |y|\\
      & \leq \left\{\begin{array}{ll}
|y| + \lambda & \mbox{ IST},\\
|y| + \sqrt{2\lambda} & \mbox{ IHT},
\end{array}\right.
\end{aligned}
\end{equation*}
which completes the proof of the lemma.
\end{proof}

Let the true signal $x^\dag$ be $s$-sparse with a support $\mathcal{A}^\dag$, i.e., $s=|\mathcal{A}^\dag|$,
and $\mathcal{I}^\dag$ the complement of $\mathcal{A}^\dag$. Recall also that the mutual coherence (MC) $\mu$ of the matrix $\Psi$
is defined by $\mu=\max_{i\neq j}|\langle \psi_i,\psi_j\rangle|$ \cite{DonohoHuo:2001}.
\begin{assumption}\label{ass:mc}
The MC $\mu$ of $\Psi$ satisfies $\mu s< 1/2.$
\end{assumption}

The proper choice of the regularization parameter $\lambda$ is essential for
successful sparse recovery. It is well known that under Assumption \ref{ass:mc}, the choice $\lambda =
O(\epsilon)$ for the $\ell_1$ penalty and $\lambda = O(\epsilon^2)$ for the $\ell_0$
penalty ensures $\|x - x^\dag\|_{\ell^\infty} = O(\epsilon)$ \cite{DonohoEladTemlyakov:2006,JiaoJinLu:2015}.
Thus we consider the following \textit{a priori} choice
\begin{equation}\label{eqn:lambda}
   \lambda^* = \left\{\begin{array}{ll}
     C_1\epsilon,\ \mbox{with } C_1 > \frac{1}{1-2\mu s}, & \mbox{for ISTC},\\
     C_0\epsilon^2,\ \mbox{with } C_0 > \frac{1}{2(1-2\mu s)^2}, & \mbox{for IHTC}.
   \end{array}\right.
\end{equation}
In practice, one may consider \textit{a posteriori} choice rules \cite{ItoJin:2014}.
Now we can state the global convergence of Algorithm \ref{alg:ithc}.
\begin{theorem}\label{thm:consoft}
Let Assumption \ref{ass:mc} hold, and $\lambda^*$ be chosen by \eqref{eqn:lambda}.
Suppose that $\lambda_0$ is large, $K_{max}\in\mathbb{N}$, and
\begin{equation*}
  \gamma\in\left\{\begin{array}{ll}
    \ [{2\mu s}/(1- 1/C_1),1), & \mbox{for ISTC},\\
    \ [(\frac{2\mu s}{1-{1}/({2C_0})^{1/2}})^2,1), & \mbox{for IHTC}.
  \end{array}\right.
\end{equation*}
Then Algorithm \ref{alg:ithc} is well-defined, and the solution $x^*$ satisfies:
\begin{itemize}
\item[(i)] $\mathrm{supp} (x^*) \subset \mathcal{A}^\dag$,
\item[(ii)] there holds the error estimate
\begin{equation*}
  \|x^*-x^\dag\|_{\ell^\infty}\leq \left\{\begin{aligned}
    ({C_1-1})\epsilon/({\mu s}), & \quad \mbox{for ISTC},\\
    (\sqrt{2C_0} -1)\epsilon/(\mu s), & \quad \mbox{for IHTC}.
  \end{aligned}\right.
\end{equation*}
\end{itemize}
Further, if $\min_{i\in\mathcal{A}^\dag} |x_i^\dag|$ is large enough, then $\mathrm{supp}(x^*)= \mathcal{A}^\dag$.
\end{theorem}
\begin{proof}
We only prove the assertion for ISTC, since that for IHTC is similar. The choice of $C_1$ in
\eqref{eqn:lambda} implies $C_1>1$ and
$\frac{2\mu s}{1-1/C_1}<1$, and thus the choice of $\gamma$ makes sense.

First we consider the inner loop at lines 5 - 7 of Algorithm \ref{alg:ithc} and
omit the index $\ell$ for notational simplicity. Let
$E^k =\|x^k - x^\dag\|_{\ell^\infty}$, and $\alpha = \frac{1-{1}/{C_1}}{\mu s}$.
Consider one IST iteration from $x^k$ to $x^{k+1}$.
The key step to the convergence proof is the following implication: with $\calA^k=\mathrm{supp}(x^k)$
\begin{equation}\label{eqn:key}
  \begin{aligned}
     &\mathcal{A}^k \subset\mathcal{A}^\dag \mbox{ and } E^k \leq \alpha\lambda \\
   \Rightarrow &\mathcal{A}^{k+1} \subset\mathcal{A}^\dag \mbox{ and } E^{k+1} \leq \alpha\gamma\lambda\quad \forall\lambda\geq \lambda^*.
  \end{aligned}
\end{equation}
Now we show this claim. It follows from \eqref{eqn:gov} and $\|\Psi_i\|=1$
the following componentwise expression for the update
\begin{equation*}
\begin{aligned}
x^{k+1}_i& =T_\lambda(x_i^k+\Psi_i^t(y-\Psi x^k))\\
& = T_\lambda(x_i^\dag + \Psi_i^t(\Psi_{\mathcal{A}^\dag\cup \mathcal{A}^k\backslash \{i\}} (x^\dag - x^k)_{\mathcal{A}^\dag\cup \mathcal{A}^k\backslash \{i\}} + \eta)).
\end{aligned}
\end{equation*}
By the hypothesis in \eqref{eqn:key}, $ \mathcal{A}^k \subset\mathcal{A}^\dag$, $E^k \leq \alpha\lambda$, $\lambda \geq \lambda^*$
and \eqref{eqn:lambda}, we deduce that for any $i \in\mathcal{I}^\dag$
\begin{equation*}
   \begin{aligned}
   &|x_i^\dag + \Psi_i^t(\Psi_{\mathcal{A}^\dag\cup \mathcal{A}^k\backslash \{i\}} (x^\dag - x^k)_{\mathcal{A}^\dag\cup \mathcal{A}^k\backslash \{i\}} + \eta)|\\
   \leq &|\Psi_i^t(\Psi_{\mathcal{A}^\dag} (x^\dag - x^k)_{\mathcal{A}^\dag}| + |\Psi_i^t\eta|\\
   \leq & \mu s E^k + \epsilon \leq (\tfrac{1}{C_1}+ \mu s\alpha) \lambda = \lambda,
  \end{aligned}
\end{equation*}
by the definition of $\alpha$, and the second inequality follows from \cite[Lemma 2.1]{JiaoJinLu:2015}.
Hence, $|x^{k+1}_i|  \leq |T_\lambda (\mu sE^k + \epsilon)| = 0$, which
implies directly $\mathcal{A}^{k+1} \subset\mathcal{A}^\dag$. Meanwhile,
under \eqref{eqn:key} and \eqref{eqn:lambda}, for any  $i\in\mathcal{A}^\dag$, by Lemma \ref{lem:thresholding}, we deduce
\begin{equation*}
 \begin{aligned}
 |x^{k+1}_i - x^\dag_i| & \leq \lambda + |\Psi_i^t(\Psi_{\mathcal{A}^\dag\backslash \{i\}} (x^\dag - x^k)_{\mathcal{A}^\dag\backslash \{i\}}| + |\Psi_i^t\eta|\\
  &\leq \lambda + \mu(s-1) E^k + \epsilon \leq \lambda + \mu s\alpha\lambda + \tfrac{1}{C_1}\lambda\\
  &= (1 + \tfrac{1}{C_1} + \alpha \mu s)\lambda  =  2\lambda \leq \alpha \gamma \lambda.
\end{aligned}
\end{equation*}
Thus we have $E^{k+1} \leq \alpha\gamma\lambda$, i.e., the claim \eqref{eqn:key} holds.

Next we prove the following assertion by mathematical induction:
for all $\ell$ with $\lambda_\ell\geq \lambda^*$, there holds
\begin{equation}\label{eqn:mathindu}
\mathrm{supp}\; x(\lambda_\ell) \subset \mathcal{A}^\dag, \quad \|x(\lambda_\ell) - x^\dag\|_{\ell^\infty}
\leq \alpha \gamma \lambda_\ell.
\end{equation}
Since $\lambda_0$ is large, it satisfies \eqref{eqn:mathindu}. Now assume \eqref{eqn:mathindu} holds for $\lambda_{\ell-1}$, i.e.,
$\mathrm{supp}\; x(\lambda_{\ell-1}) \subset \mathcal{A}^\dag$ and $\|x(\lambda_{\ell-1}) - x^\dag\|_{\ell^\infty}
\leq \alpha \gamma \lambda_{\ell-1}$. When Algorithm \ref{alg:ithc} runs lines 3 - 7 for $\lambda_\ell$, since $x^0 = x(\lambda_{\ell-1})$, then we have
$\mathcal{A}^0 \subset\mathcal{A}^\dag$ and $ E^0 \leq \alpha\lambda_\ell.$
From \eqref{eqn:key}, we obtain that for all $k \geq 1$,
$\mathcal{A}^k \subset\mathcal{A}^\dag \mbox{ and } E^k \leq \alpha\gamma\lambda_\ell.$
 In particular, if we choose $k = K_{max}$, then \eqref{eqn:mathindu} holds for $\lambda_\ell$.
When Algorithm \ref{alg:ithc} terminates for some $\lambda_\ell < \lambda^*$, then $\lambda_{\ell-1} \geq \lambda^*$ and $x^* = x(\lambda_{\ell-1})$. From \eqref{eqn:mathindu}
we have $\mathrm{supp}\; x^* \subset \mathcal{A}^\dag$ and $\|x^* - x^\dag\|_{\ell^\infty} \leq \alpha
\lambda^* = (C_1 -1)\epsilon/(\mu s)$. Likewise, if $\min_{i\in\calA^\dag}|x_i|>(C_1-1)\epsilon/(\mu s)$,
property (ii) implies $\mathrm{supp}(x^*)=\calA^\dag$.

Last, we briefly discuss IHTC. For the choice $C_0$ in \eqref{eqn:lambda}, $\gamma\in
[(\frac{2\mu s}{1-{1}/({2C_0})^{1/2}})^2,1)$ makes sense. With
$\alpha = \frac{1-{1}/({2C_0})^{1/2}}{\mu s}$, a similar argument yields
\begin{equation*}
  \begin{aligned}
    &\mathcal{A}^k \subset\mathcal{A}^\dag \mbox{ and } E^k \leq \alpha\sqrt{2\lambda}\\
    \Rightarrow & \mathcal{A}^{k+1} \subset\mathcal{A}^\dag \mbox{ and } E^{k+1} \leq \alpha\sqrt{2\gamma\lambda}.
  \end{aligned}
\end{equation*}
The rest follows like before, and thus it is omitted.
\end{proof}

\begin{remark}
The proof works for any choice $K_{max}\geq 1$, including $K_{max}=1$. In
practice, we fix it at $K_{max}=5$. This together with Theorem \ref{thm:consoft} allows estimating
the complexity of Algorithm \ref{alg:ithc}. At
each iteration, one needs to compute matrix-vector product $\Psi x$ and $\Psi^t y$, and
for each $\lambda$, the number of iterations is bounded by  $K_{max}$. The overall
cost depends on the decreasing factor $\gamma$ by $O(\frac{\ln \lambda^*}{\ln \gamma}np) = O(\frac{\ln \epsilon}{\ln \gamma} np)$.
\end{remark}
\begin{remark}\label{rmk:mc}
Conditions similar to Assumption \ref{ass:mc} have been widely used in the literature, for analyzing OMP
\cite{TroppGilbert:2007,CaiWang:2011,DonohoEladTemlyakov:2006} (with $(2s-1)\mu\leq 1$) and
for bounding the estimation error of Lasso \cite{Lounici:2008,Zhang:2009} (with $7s\mu<1$ and $4s\mu\leq 1$).
Thus Assumption \ref{ass:mc} is fairly standard. Examples of matrices with small MC $\mu$ include
that formed by equiangular tight frame and random subgaussian matrices \cite{FoucartRauhut:2013}.
Further, we note that other similar conditions, e.g., restricted eigenvalue condition and RIP conditions, were also used
to derive error bounds of the type $\|x - x^\dag\|_2 = O(\epsilon)$ for proximal gradient
homotopy algorithms \cite{XiaoZhang:2013,Agarwal:2012} and Greedy methods, e.g., CoSaMP
\cite{NeedellTropp:2009}, NIHT \cite{BlumensathDavis:2010} and CGIHT \cite{BlanchardTannerWei:2015}.
\end{remark}

\section{Numerical Results and Discussions}

Now we present numerical examples to show the convergence and
the performance of Algorithm \ref{alg:ithc}. First, we give
implementation details, e.g., data generation, parameter setting for the algorithm. Then our method is
compared with several state-of-the-art algorithms in terms of reconstruction error and recovery ability
via phase transition.

\subsection{Implementation details}
Following \cite{nesta}, the signals $x^\dag$ are chosen as $s$-sparse with a dynamic range
$DR := \max\{|x^\dag_{i}|:x^\dag_{i}\neq 0\}/\min\{|x^\dag_{i}|:x^\dag_{i}\neq 0\}.$
The matrix $\Psi\in\mathbb{R}^{n\times p}$ is chosen to be either random Gaussian matrix,
or random Bernoulli matrix, or  the product of a partial FFT matrix and inverse
Haar wavelet transform. Under proper conditions, such matrices satisfy Assumption
\ref{ass:mc}. The noise  $\eta$ has entries following i.i.d. $ N(0,\sigma^2)$.

We fix the  algorithm parameters as follows: $\lambda_0=\|\Psi^t y\|_{\infty}$ and  $\lambda_0=\|\Psi^t y\|_{\infty}^2/2$
for ISTC and IHTC, respectively \cite{FanJiaoLu:2014,JiaoJinLu:2015}, decreasing factor $\gamma=0.8$.
 Since the optimal $\lambda^*$ depends on the  noise level $\epsilon$, which
is often unknown in practice, we predefine a path $\Lambda = \{\lambda_\ell\}_{\ell=0}^N$ with $\lambda_\ell = \lambda_0\gamma^\ell$
and  $N=100$. Then we run Algorithm \ref{alg:ithc} on the path $\Lambda$ and select the optimal
$\lambda^*$ by Bayesian information criterion \cite{FanJiaoLu:2014}.
All the computations were performed on an eight-core desktop with 3.40 GHz and 12 GB RAM using \texttt{MATLAB} 2014a.
The \texttt{MATLAB} package \texttt{ISHTC} for reproducing all the numerical results can be found at
\url{http://www0.cs.ucl.ac.uk/staff/b.jin/companioncode.html}.

First we illustrate Theorem \ref{thm:consoft} by examining the influence of sparsity level $s$,  coherence $\mu$ and noise level $\sigma$
on IHTC recovery on three settings ($n  = 500$, $p= 1000$, $DR = 100$):
\begin{enumerate}
\item[(a)] random Gaussian $\Psi$, $\sigma =$1e-2, $ s = 10:10:100$.
\item[(b)]  random Gaussian $\Psi$, $ s = 50$, $\sigma=$1e-4,1e-3,1e-2,1e-1,1.
\item[(c)] $\Psi$ is random Gaussian with correlation, where the parameter $\nu$ controls the coherence $\mu$ (see \cite[Sect. 5.1]{JiaoJinLu:2013} for details). In general a larger parameter $\nu$ gives a larger $\mu$ (a typical example: $\mu = 0.19$ for $\nu = 0$; $\mu = 0.33$ for $\nu = 0.15$; $\mu = 0.56$ for $\nu = 0.3$ and $\mu = 0.74$ for $\nu = 0.5$). We choose $\nu = 0:0.05:1$, $s = 10$, $\sigma  =$ 1e-3.
\end{enumerate}
The results in Fig. \ref{fig:inf} are computed from 100 independent realizations.
It is observed that when the sparsity level $s$ and noise level $\sigma$ and incoherence $\nu$ are
small, IHTC recovers the  exact support
with high probability as implied by Theorem \ref{thm:consoft}.

\begin{figure}[htb!]
  \centering
  \begin{tabular}{ccc}
   \includegraphics[trim = 1cm 0cm 2cm 0.8cm, clip=true,width=.15\textwidth]{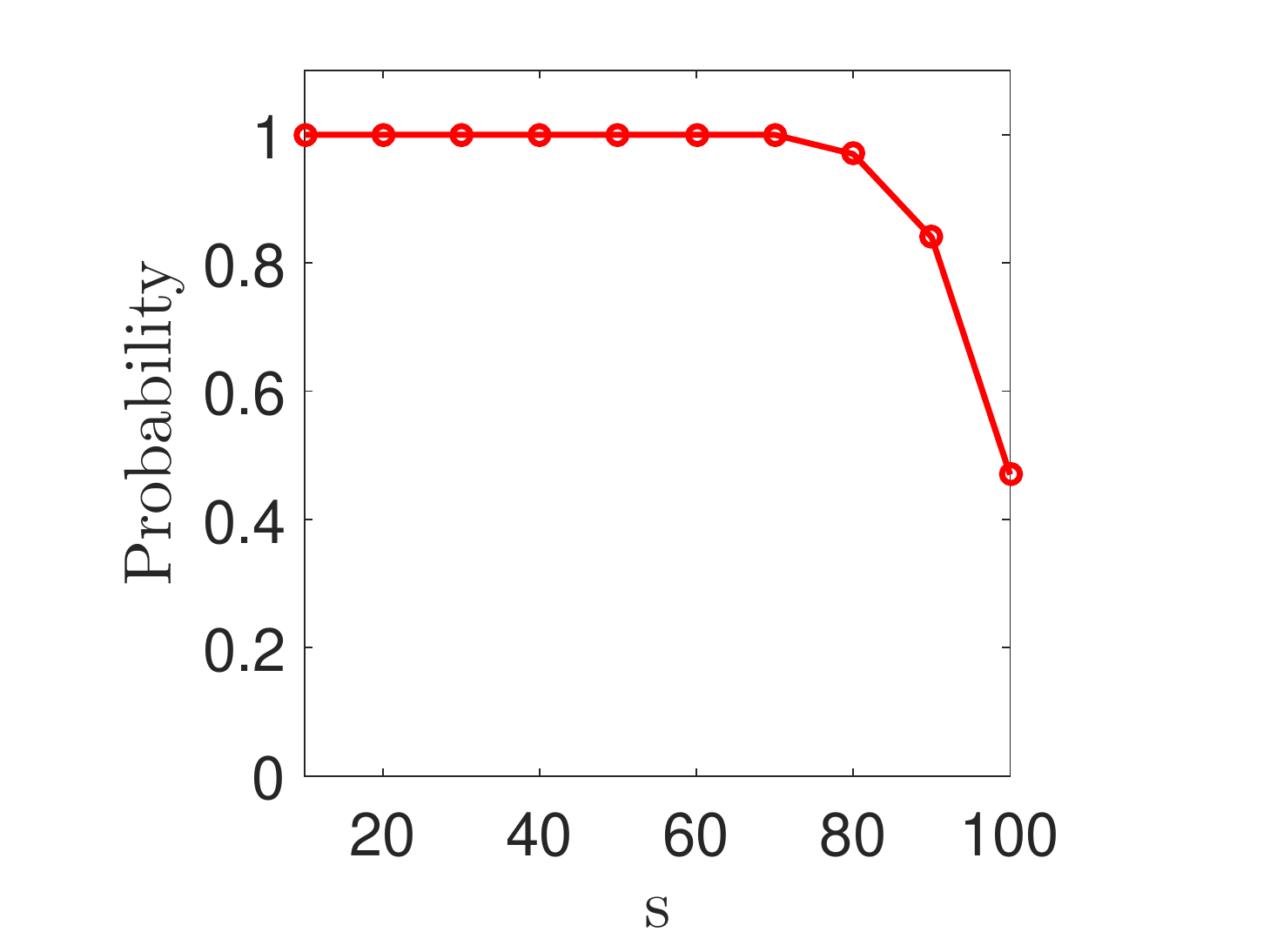} &
   \includegraphics[trim = 1cm 0cm 2cm 0.8cm, clip=true,width=.15\textwidth]{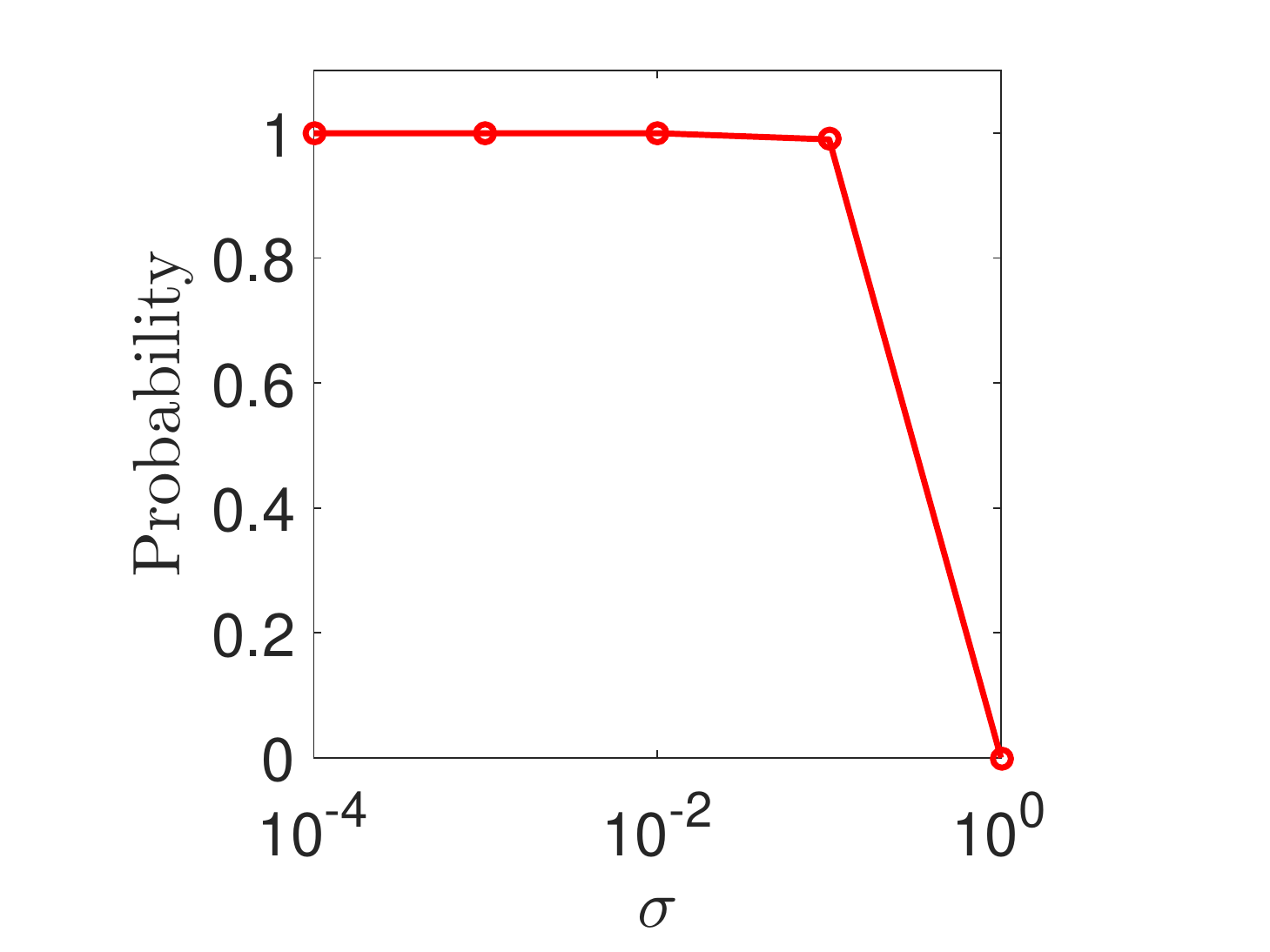} &
   \includegraphics[trim = 1cm 0cm 2cm 0.8cm, clip=true,width=.15\textwidth]{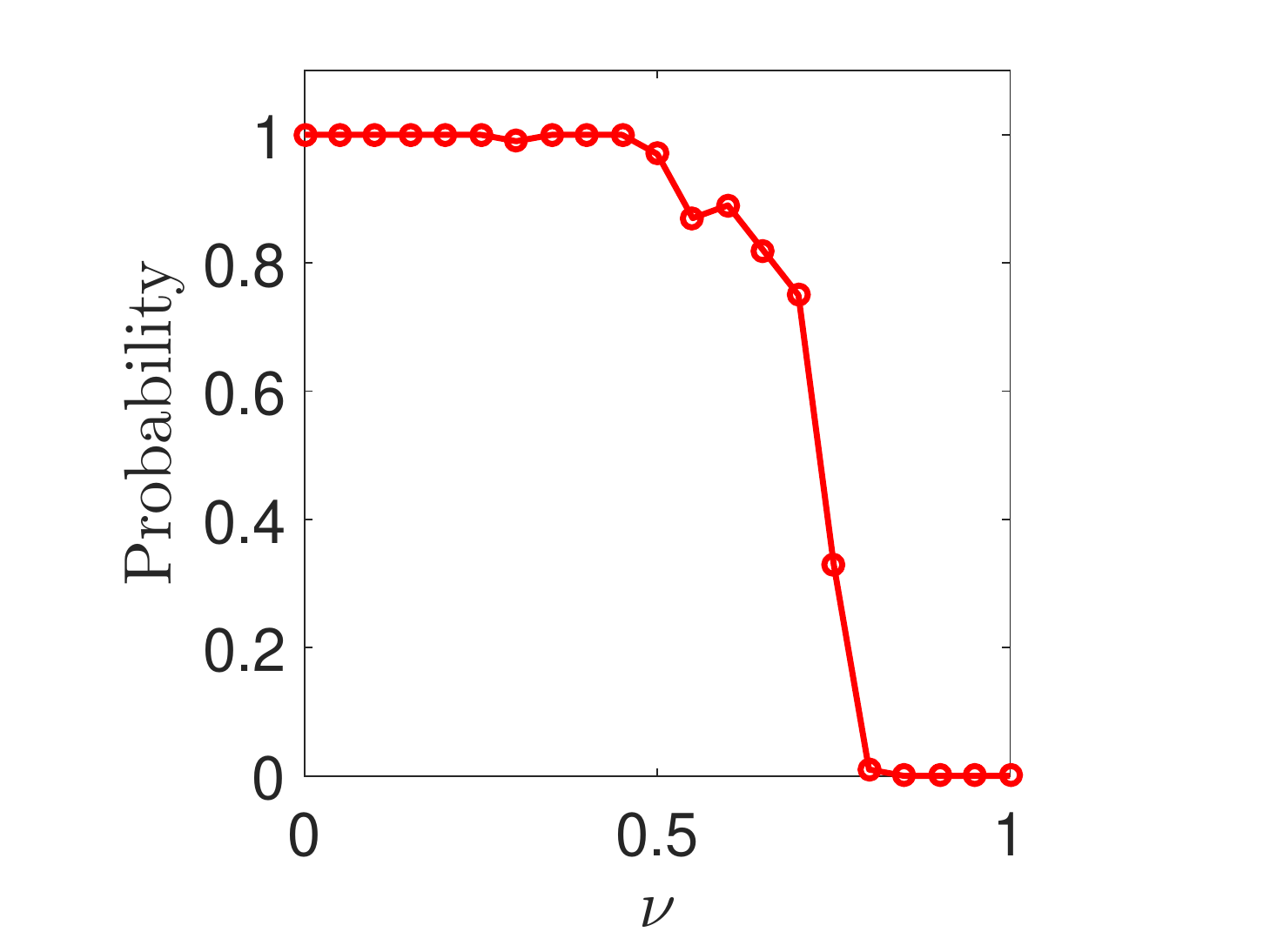}\\
   (a) $s$ & (b) $\sigma$ &   (c) $\nu$\\
  \end{tabular}
   \caption{The exact support recovery probability v.s. $s$, $\sigma$ and $\nu$}\label{fig:inf}
\end{figure}

\subsection{Comparison of ISTC with $\ell^1$ solvers}
Now we compare ISTC with four state-of-the-art $\ell^1$ solvers:
GPSR \cite{gpsr} (\url{http://www.lx.it.pt/mtf/GPSR/}),
SpaRSA \cite{sparsa} (\url{http://www.lx.it.pt/mtf/SpaRSA/}), proximal-gradient
homotopy method (PGH)\cite{XiaoZhang:2013} (\url{https://www.microsoft.com/en-us/download/details.aspx?id=52421}),
and FISTA \cite{fista} (implemented as
\url{https://web.iem.technion.ac.il/images/user-files/becka/papers/wavelet_FISTA.zip})\footnote{All the codes were last accessed
on February 23, 2017.}.

The numerical results  (CPU time, number of matrix-vector multiplications (nMV),  relative $\ell_{2}$  error (Re$\ell_{2}$), and absolute $\ell_{\infty}$
 error (Ab$\ell_{\infty}$)) are computed from 10 independent realizations of for random
 Bernoulli  sensing matrices with different parameter tuples $(n,p,s,DR,\sigma)$ are shown
in Tables \ref{tab:timeerrorb}. It is observed that ISTC
yields reconstructions that are comparable with that by other methods but
at least two to three times faster. Further, it scales well with the problem size $p$.

\begin{table}[htb!]
\centering
\caption{Numerical results (CPU time and errors), with random Bernoulli
$\Psi$, of size $p =$ 10000, 18000, $n = \lfloor p/4\rfloor$,
$s= \lfloor n/40\rfloor$, with $DR=100$ and $\sigma=\mbox{5e-2}$.}\label{tab:timeerrorb}
\begin{tabular}{ccccccp{0.3cm}p{0.3cm}c}
\hline\hline
\multicolumn{1}{c}{$p$} & \multicolumn{1}{c}{method} & \multicolumn{1}{c}{time (s)} & \multicolumn{1}{c}{nMV}& \multicolumn{1}{c}{Re$\ell^2$} & \multicolumn{1}{c}{Ab$\ell^{\infty}$} \\
 \hline
                  &ISTC              &  1.0   &58 &4.21e-3         &2.66e-1            \\
                  &PGH              &  1.7    &419 &4.14e-3         &2.66e-1      \\
    $10000$        &SpaRSA           &  3.4   &302&4.13e-3         &2.63e-1            \\
                  &GPSR             &  3.0    &256 &4.25e-3         &2.71e-1          \\
                   &FISTA            &  5.3   &505 & 4.30e-3               & 2.65e-1                 \\
  \hline
                  &ISTC              &  3.3   &58&4.34e-3         &2.88e-1            \\
                  &PGH              &  5.6    &443 &4.25e-3         &2.85e-1      \\
    $18000$       &SpaRSA           &  11.4   &309 &4.25e-3         &2.84e-1            \\
                  &GPSR             &  9.5    &258 &4.36e-3         &2.91e-1          \\
                  &FISTA             &  17.2  &506 & 4.40e-3               & 2.74e-1                 \\
\hline\hline
\end{tabular}
\end{table}

Next, we compare the empirical performance of ISTC with other methods
by their phase transition curves in the  $\rho$-$\delta$ plane,
with $\rho = s/n$ and $\delta = n/p$. When computing the curves, we
fix the dimension $p = 1000$, and partition the range $(\delta,\rho)
\times[0.1,1]^2$ into a $30\times 30$ equally spaced grid, and run
100 independent simulations at each grid point.
The $s$-sparse signal $x^{\dag}\in\mathbb{R}^p$, matrix
$\Psi\in\mathbb{R}^{n\times p}$, and data $y\in\mathbb{R}^n$ are generated as
\cite[Fig. 13]{DonohoTsaig:2008}. Fig. \ref{fig:phasesoft} plots the logistic regression curves
identifying the $90\%$ success rate for the algorithms.
IHTC exhibits similar phase transition behavior as
other methods.

\begin{figure}[h]
\centering
\includegraphics[trim = 0cm 0cm 0cm 0.5cm, clip=true,width=5cm]{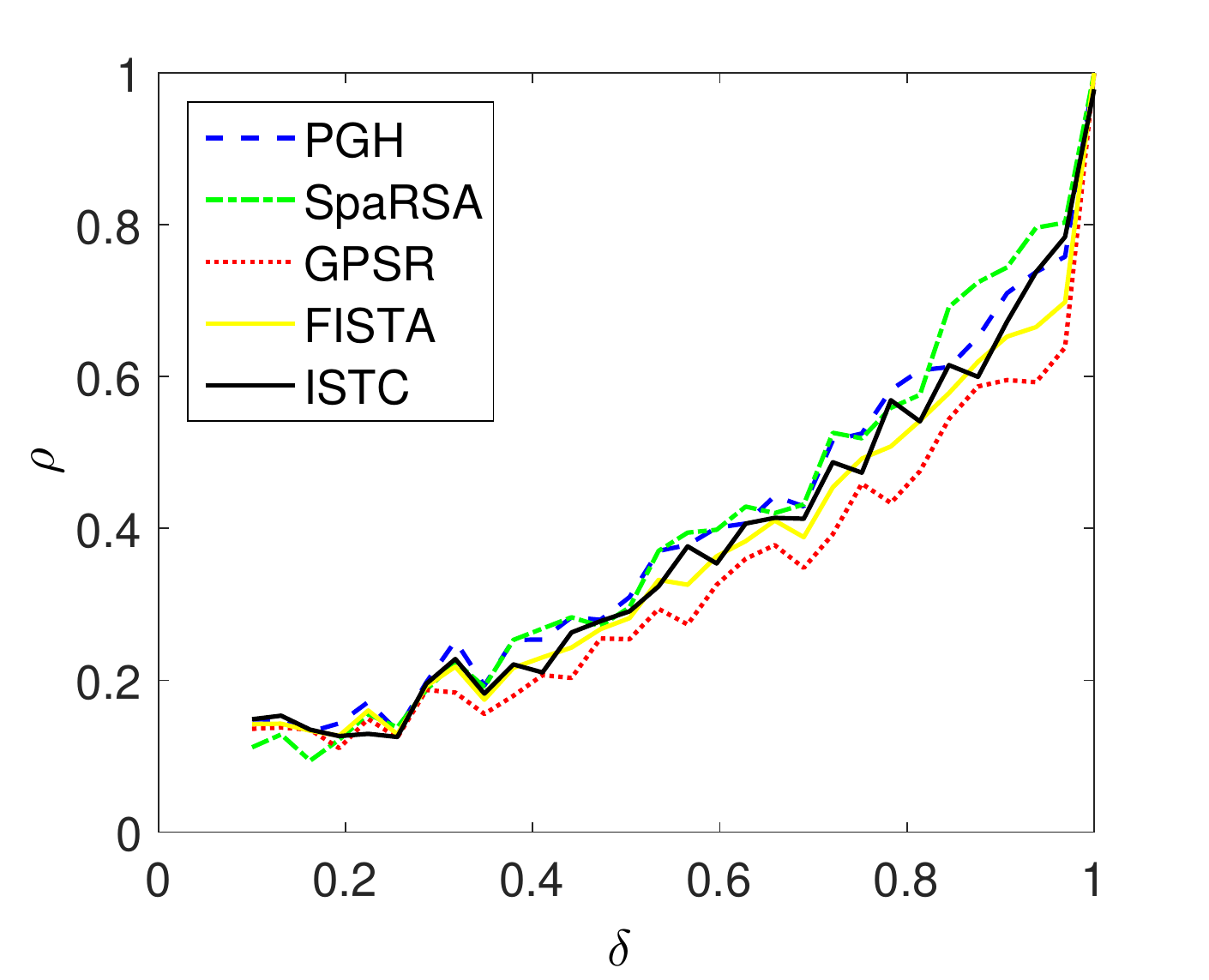}
\caption{The empirical phase transition curves for ISTC, PGH, SpaRSA and GPSR, with $\rho = s/n$ and $\delta = n/p$.}\label{fig:phasesoft}
\end{figure}

\subsection{Comparison of IHTC with greedy solvers}
Now we compare IHTC with four state-of-the-art greedy methods for the $\ell^0$ problem, to
recover 1D signal and benchmark MRI image. These methods include
OMP \cite{TroppGilbert:2007} (\url{https://sparselab.stanford.edu/SparseLab_files/Download_files/SparseLab21-Core.zip}),
{normalized  IHT (NIHT) \cite{BlumensathDavis:2010}} (\url{http://www.gaga4cs.org/}), CoSaMP \cite{NeedellTropp:2009} (\url{http://mdav.ece.gatech.edu/software/SSCoSaMP-1.0.zip}), and conjugate gradient IHT (CGIHT) \cite{BlanchardTannerWei:2015} (\url{http://www.gaga4cs.org/}).

The underlying 1D signal and 2D MRI image  are compressible under
a  wavelet basis. Thus, the data can be chosen as the wavelet coefficients
sampled by the product of a partial FFT matrix and inverse Haar wavelet transform.
For the 1D signal, the matrix $\Psi$ is of size $665\times 1024$, and consists
of applying a partial FFT and an inverse two level Harr wavelet transform.
The signal under wavelet transform has $247$ nonzeros, and $\sigma=\mbox{1e-4}$. The results are shown in
Fig. \ref{fig:1} and Table \ref{tab:3}. The reconstruction by IHTC
is  visually more appealing than that of the others, cf. Fig. \ref{fig:1}.  The results by AIHT
and CoSaMP suffer from pronounced oscillations. This is further confirmed by the PSNR value defined by
$\mathrm{PSNR}=10\cdot \log\frac{V^2}{\rm MSE}$,
where $V$ is the maximum absolute value of the true signal, and MSE is the mean
squared error of the reconstruction. Table \ref{tab:3} also presents the CPU time
of the 1D example, which shows clearly that IHTC is the fastest one.

For the 2D MRI image, the matrix $\Psi$ amounts to a partial FFT and an inverse
wavelet transform, and it has a size $34489\times 262144$. The image under eight level Haar wavelet transformation  has
$7926$ nonzero entries and $\sigma=\mbox{3e-2}$. The numerical results are
shown in Fig. \ref{fig:2} and Table \ref{tab:4}.  All  $\ell^0$ methods produce comparable
results, but the IHTC is fastest.

\begin{figure}[h]
\centering
\includegraphics[trim = 1cm 2cm 1cm 0cm, clip=true,width=8cm]{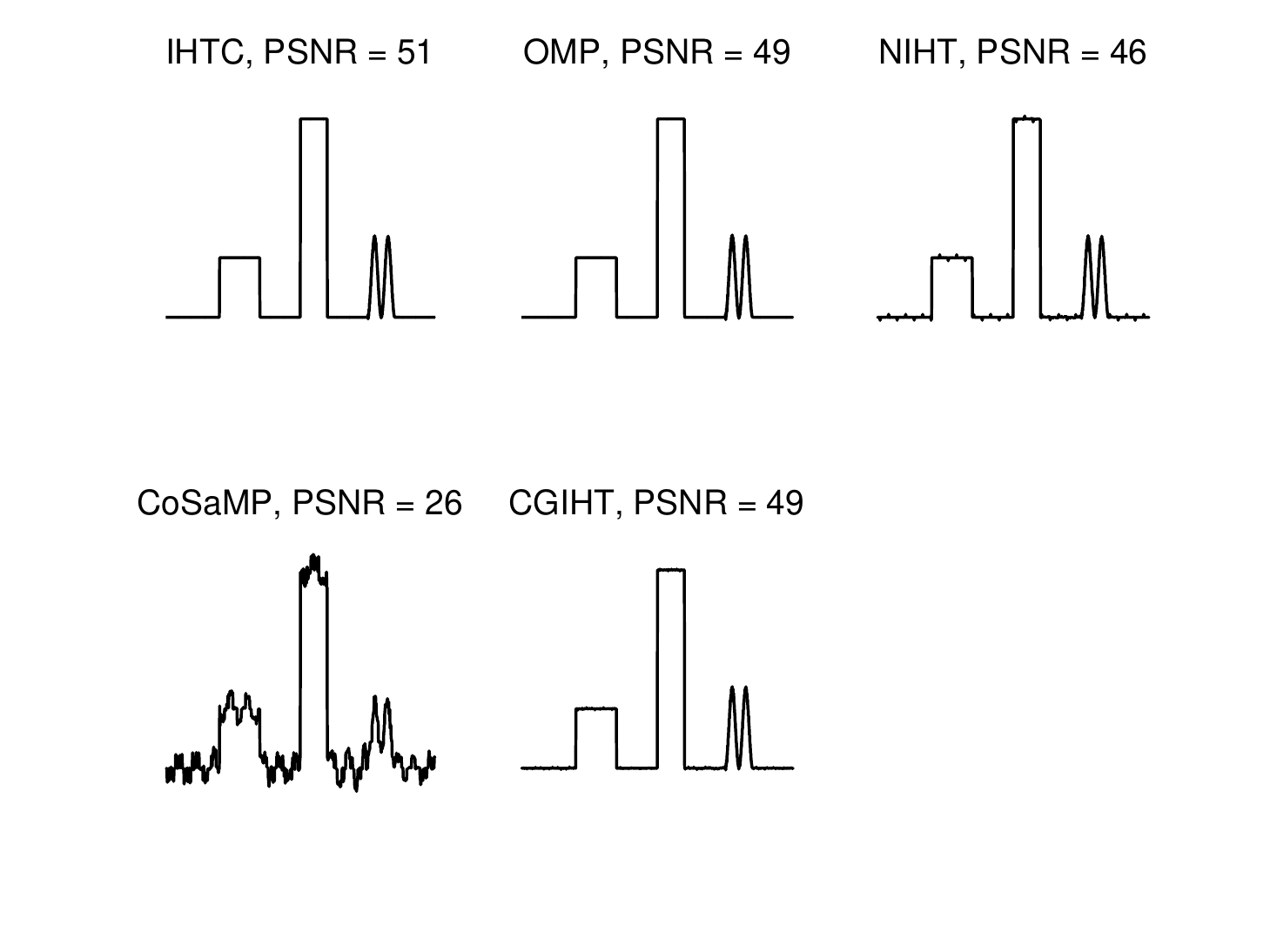}
\caption{Reconstructed signals and their PSNR values  }\label{fig:1}
\end{figure}

\begin{figure}[hbt!]
\centering
\includegraphics[trim = 1cm 1cm 1cm 0cm, clip=true,width=8cm]{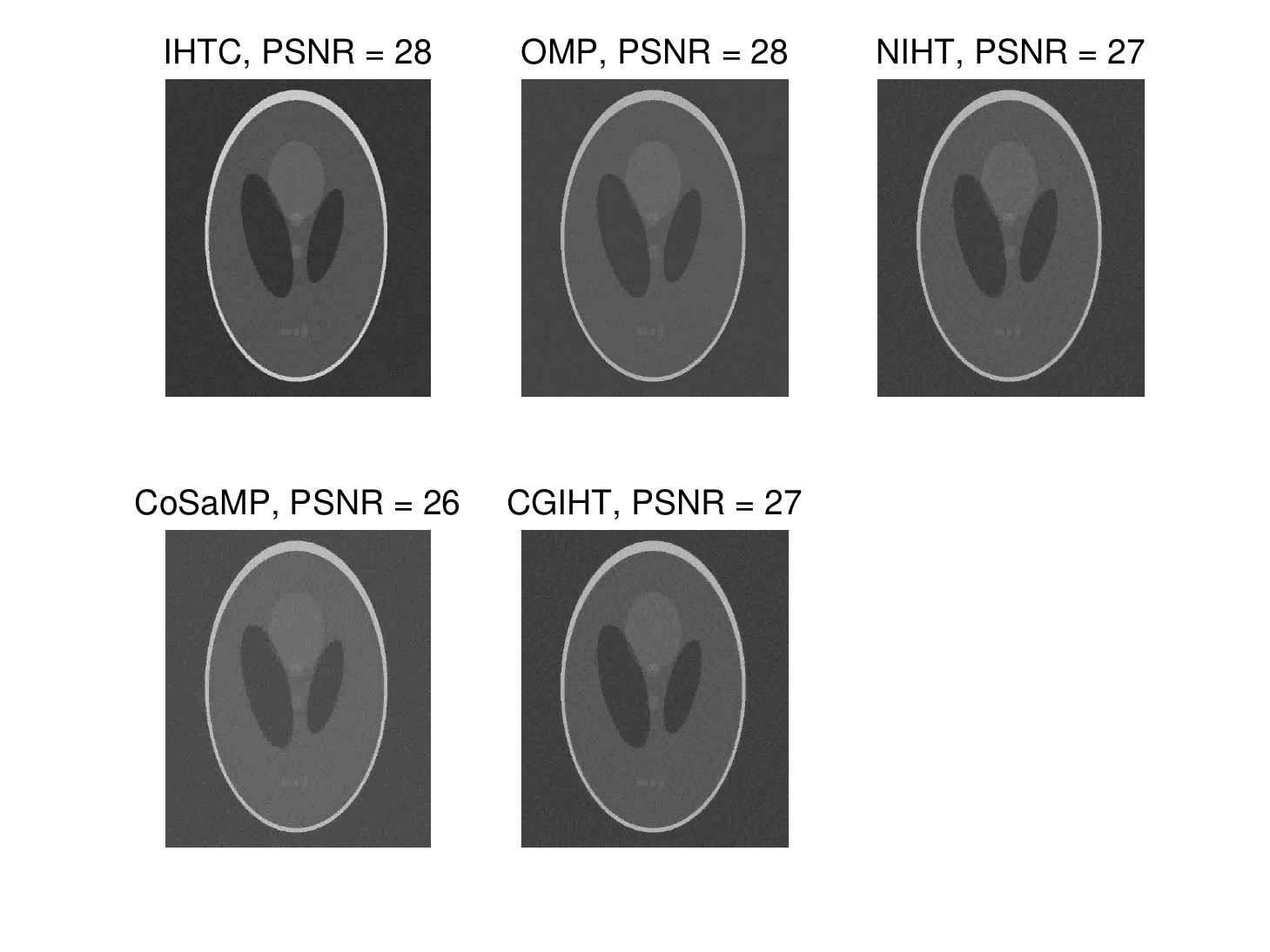}
\caption{Reconstructed MRI images and their PSNR values  }\label{fig:2}
\end{figure}

\begin{table}
{Left: 1D signal with $n=665$, $p=1024$, $s=247$, and $\sigma$=1e-4. Right:
2D image with $n=34489$, $p=262144$, $s=7926$, and $\sigma$=3e-2. }
\parbox{.22\textwidth}{
\centering
\caption{1D signal\label{tab:3}}
 \begin{tabular}{ccccc}
 \hline
    method   &CPU time  &PSNR      \\
 \hline
  IHTC      & 0.41  &51     \\
  OMP  & 1.20  &49     \\
  NIHT & 0.96  &46      \\
  CoSaMP       & 0.49  &26    \\
  CGIHT & 0.98 &49\\
  \hline
  \end{tabular}
}\quad\quad
\parbox{.22\textwidth}{
\centering
\caption{2D image\label{tab:4}}
 \begin{tabular}{ccccc}
 \hline
    method   &CPU time  &PSNR      \\
 \hline
  IHTC      & 6.1  &28     \\
  OMP  & 932  &28     \\
   NIHT    & 9.4  &27       \\
  CoSaMP      & 14.3  &26    \\
  CGIHT  &7.9 & 27\\
  \hline
  \end{tabular}
}
\end{table}

Next, we compare the empirical sparse recovery performance of IHTC with these greedy methods
by means of phase transition curves in the $\rho$-$\delta$ plane, with $\rho = s/n$ and $\delta = n/p$.
When computing the curves, we fix the dimension $p = 1000$, partition the range $(\delta,\rho)\in
[0.1,1]^2$ into a $90\times 90$ uniform grid, and run 100 independent simulations
at each grid point. Like before, the $s$-sparse signal $x^{\dag}\in\mathbb{R}^p$, matrix
$\Psi\in\mathbb{R}^{n\times p}$ and data $y\in\mathbb{R}^n$ are generated as
\cite[Fig. 13]{DonohoTsaig:2008}.  Fig. \ref{fig:phasehard} plots the logistic regression curves
identifying the $90\%$ success rate for the algorithms. IHTC exhibits
comparable phase transition phenomenon with other greedy methods, whereas CoSaMP performs slightly worse than others.

\begin{figure}[h]
\centering
\includegraphics[trim = 0cm 0cm 0cm .5cm, clip=true,width=5cm]{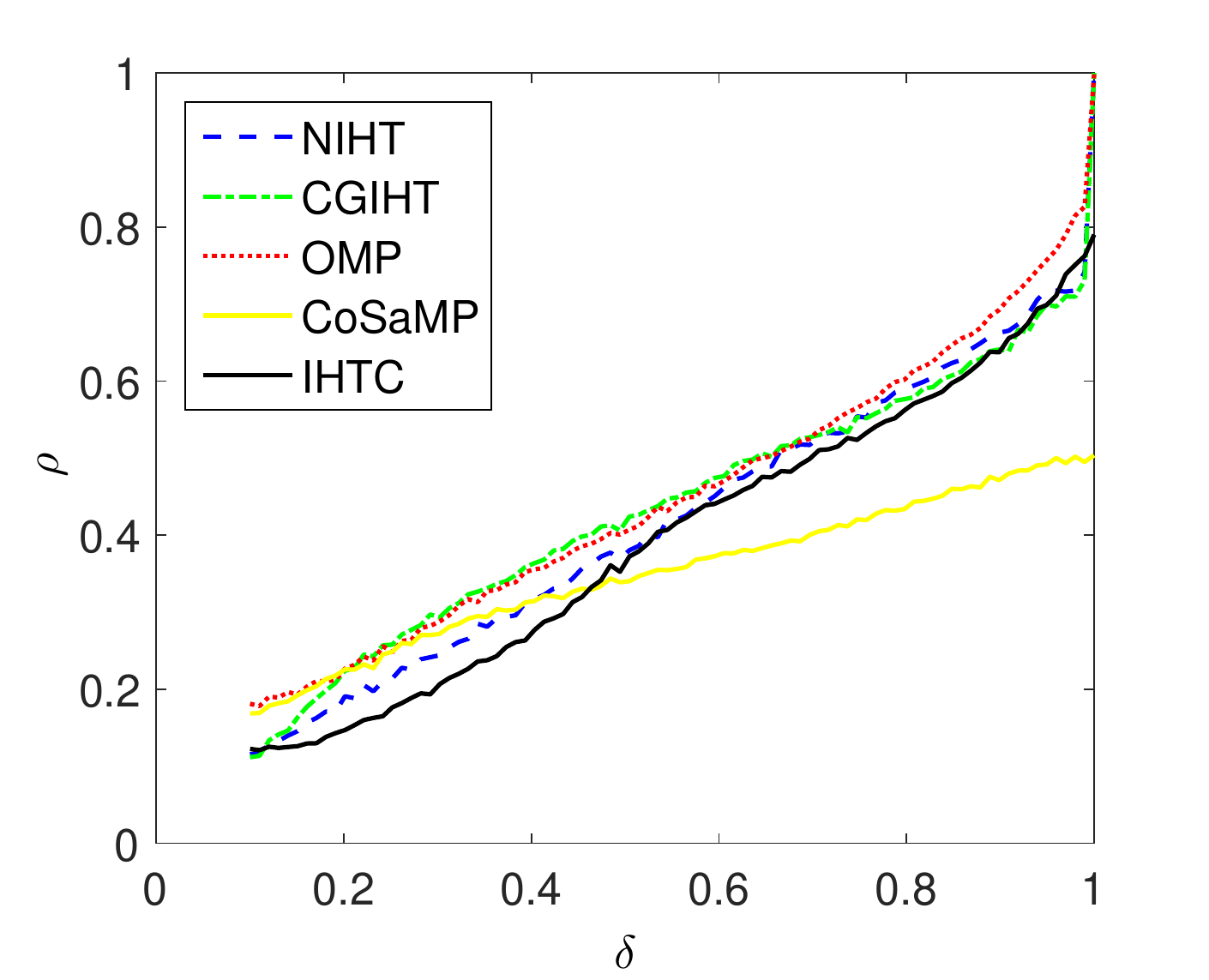}
\caption{The empirical phase transition curves of IHTC, OMP, CoSaMP, NIHT and CGIHT, with $\rho = s/n$ and $\delta = n/p$.}\label{fig:phasehard}
\end{figure}

\section{Conclusion}
In this paper, we analyze an iterative soft / hard thresholding algorithm with homotopy
continuation for sparse recovery from noisy  data.
Under standard regularity condition and sparsity assumptions,   sharp reconstruction errors
can be obtained  with an iteration complexity
$O(\frac{\ln \epsilon}{\ln \gamma} np)$. Numerical results indicated its competitiveness
with state-of-the-art sparse recovery algorithms. The results can be extended to other penalties, e.g.,
MCP \cite{Zhang:2010} or SCAD \cite{FanLi:2001}.

\section*{Acknowledgements}
The authors thank anonymous referees for their
helpful comments.
The research of Y. Jiao is partially supported by
National Science Foundation of  China (NSFC) No. 11501579 and National Science Foundation of  Hubei Province No. 2016CFB486, B. Jin by
EPSRC grant EP/M025160/1, and X. Lu
by NSFC Nos. 11471253 and 91630313.

\bibliographystyle{IEEEtran}
\bibliography{ishtc}
\end{document}